\documentclass[a4paper,12pt]{amsart}
\usepackage{amsmath,amssymb,amsfonts,amsthm}
\usepackage{bbm}
\usepackage{mathrsfs}

\newtheorem{thm}{Theorem}[section]
\newtheorem*{thm*}{Theorem}
\newtheorem{lem}[thm]{Lemma}
\newtheorem*{lem*}{Lemma}
\newtheorem{cor}[thm]{Corollary}
\newtheorem*{cor*}{Corollary}

\newtheorem*{prop*}{Proposition}
\theoremstyle{definition} 
\newtheorem{defn}[thm]{Definition}
\newtheorem*{defn*}{Definition}
\theoremstyle{remark}
\newtheorem{rem}[thm]{Remark}
\newtheorem*{rem*}{Remark}
\newtheorem{example}[thm]{Example}
\newtheorem*{example*}{Example}

\newtheorem*{que*}{Question}

\newcommand{\N}{\mathbb N}

\newcommand{\C}{\mathbb C}

\newcommand{\set}[1]{\left\{#1\right\}}
\newcommand{\eps}{\varepsilon}

\newcommand{\ind}{\mathbbm{1}}

\renewcommand{\tilde}{\widetilde}

\newcommand{\pnorm}[2]{{\left\lVert{#2}\right\rVert}_{#1}}

\DeclareMathOperator{\cllin}{\overline{lin}}

\newcommand{\A}{\mathscr{A}}
\newcommand{\F}{\mathcal{F}}
\newcommand{\G}{{\mathcal G}}

\newcommand{\X}{X\times[\,0,1]}

\newcommand{\dist}{\mathrm{dist}}

\newcommand{\DS}{\mathcal{DS}}

\renewcommand{\epsilon}{\varepsilon}

\renewcommand{\leq}{\leqslant}
\renewcommand{\geq}{\geqslant}
\newcommand{\symdiff}{\triangle}

\title{Doubly stochastic operators with zero entropy}
\author{Bartosz Frej}
\author{Dawid Huczek}
\address{Faculty of Pure and Applied Mathematics,
Wroc{\l}aw University of Science and Technology, 
Wybrze\.{z}e Wyspia\'{n}skiego 27,
50-370 Wroc{\l}aw, Poland.}
\email{Bartosz.Frej@pwr.edu.pl; Dawid.Huczek@pwr.edu.pl}

\thanks{Research of both authors is supported from resources for science in years 2013-2018 as research project (NCN grant 2013/08/A/ST1/00275, Poland)}

\begin{document}
\begin{abstract}
We study doubly stochastic operators with zero entropy. We generalize three famous theorems: the Rokhlin's theorem on genericity of zero entropy, the Kushnirenko's theorem on equivalence of discrete spectrum and nullity and the Halmos-von Neumann's theorem on representation of  maps with discrete spectrum as group rotations.
\end{abstract}

\maketitle

%====================== INTRO ON OPERATOR ENTROPY ================================

\section{Introduction}
Let $\mu$ be a probability measure on a measurable space $(X,\Sigma)$.
By a \emph{doubly stochastic operator} (also called \emph{bistochastic} or \emph{Markov operator}) we mean a linear operator $T$ on the space $L^1(\mu)$ of integrable functions, which fulfills the following conditions:
\begin{itemize}
	\item[(i)] $Tf$ is positive for every positive $f \in L^1(\mu)$,
	\item[(ii)] $T\ind=\ind$ (where $\ind(x)=1$ for all $x\in X$),
	\item[(iii)] $\int Tf\,d\mu=\int f\,d\mu$ for every $f \in L^1(\mu)$.
\end{itemize}
For each $p>1$ the space $L^p(\mu)\subset L^1(\mu)$ is invariant under the action of a doubly stochastic operator $T$. Conversely, if $T:L^p(\mu)\to L^p(\mu)$ is doubly stochastic, then it can be uniquely extended to an operator on $L^q(\mu)$ for any $1\leq q <p$, with no harm to the above properties. Therefore, we can study doubly stochastic operators on space $L^p(\mu)$ for any $p\geq 1$, in particular on the space $L^2(\mu)$ of square integrable real functions. Note that such operators preserve conjugacy, hence transform real functions into real functions and, conversely, the action of $T$ is determined by its action on the space of real functions. The set of all doubly stochastic operators on $L^1(\mu)$ will be denoted by $\DS(\mu)$.

A special role is played by operators which are induced by \emph{transition probabilities (or probability kernels)}, i.e., operators defined by the formula
\[
Tf(x) = \int f(y)\,P(x,dy),
\]
where $P:X\times \Sigma \to [0,1]$ is a transition probability.
It is known that on a standard probability space every doubly stochastic operator is of this form.

In the current paper we succeed in generalizing three famous theorems on dynamical systems with zero entropy to doubly stochastic operators with zero entropy. The first one is the Rokhlin's theorem on genericity of systems with entropy zero in the weak topology on the set of all measure preserving maps on $(X,\Sigma,\mu)$. In \cite{V} A.Vershik asked if the same holds for entropy of Markov operators, whatever the definition of entropy can be in this case. We prove that, in the strong operator topology, the set of doubly stochastic operators with zero entropy is residual in the set of all doubly stochastic operators on $L^1(\mu)$. Recall that the \emph{strong operator topology} is just the topology of pointwise convergence, i.e., $T_n$ converges to $T$ in the strong operator topology if $\lVert T_nf-Tf \rVert$ converges to zero for each $f\in L^1(\mu)$.  It is known that in pointwise case, the weak topology on the space of transformations coincides with the strong operator topology on the set of their Koopman operators. Our result (thm. \ref{new_Rokhlin}) is thus a true analog of Rokhlin's result.

The second aim is to give an analog of the Kushnirenko's theorem, which characterizes null transformations, i.e., maps with zero sequence entropy along all sequences, as maps with discrete spectrum. In operator case this equivalence is no longer valid, because null operators need not have linearly dense set of eigenfunctions. Nevertheless, we present an appropriate modification, using the Jacobs-de Leeuw-Glicksberg decomposition associated to an operator (thm. \ref{thm:main}). The description of null operators naturally leads us to generalization of the Halmos-von Neumann theorem (thm. \ref{new_HvN}), which in this formulation becomes the characterization of nullity (not the discrete spectrum).

Defining the entropy we will follow the ideas which first appeared in~\cite{DF}. To be more specific, we put the finite collections of functions in the role that finite partitions played in the classic definition of entropy. We define their join and the entropy of a single collection, and then we calculate the entropy of $T$ with respect to a collection as an appropriate limit and, finally, take supremum over all such collections. The details are given in the forthcoming section.

\section{Definition of entropy}
By a \emph{collection} of functions we mean a finite sequence of measurable functions with range in unit interval. 
For a function $f \colon X\to [0,1]$ let 
\[
A_f = \set{(x,t)\in\X : t\le f(x)}
\]
and let $\A_f$ be the partition of $X\times [0,1]$ consisting of $A_f$ and its complement. 
For a collection $\F$ we define $\A_\F = \bigvee_{f\in\F} \A_f.$ If $\F\vee\G$ denotes a concatenation of collections $\F$ and $\G$ then, clearly, $\A_{\F\vee\G}=\A_\F \vee \A_\G$.
To shorten the notation we write $T^n\F=\set{T^nf \colon f\in\F}$ and  $\F^n_T$ for the concatenation of $\F$, $T\F$,... $T^{n-1}\F$.

For two collections of measurable functions $\mathcal F=\{f_1,...,f_r\}$ and 
$\mathcal G=\{g_1,...,g_{r'}\}$, $r'\leq r$, we define their $L^1$-distance $\dist(\mathcal F, \mathcal G) $ by a formula
\[
\dist(\mathcal F, \mathcal G)= 
\min_\pi\left\{\max_{1\leq i\leq r}\int |f_i-g_{\pi(i)}|\ d\mu\right\},
\]
where the mi\-ni\-mum ranges over all per\-mu\-ta\-tions $\pi$ 
of a set $\{1,2,\dots r\}$
and where~$\mathcal G$ is considered an $r$-element collection by setting 
$g_i \equiv 0$ for $r'<i\leq r$.

Let $\lambda$ mean the Lebesgue measure on the unit interval. 
We define
\begin{eqnarray*}
&& 	H(\F) = H_{\mu\times\lambda} (\A_\F) = -\sum_{A\in\A_\F} 
	(\mu\times\lambda)(A)\cdot \log(\mu\times\lambda)(A),\\
&&	h(T,\F) =  \lim_{n\to\infty} \frac1n H(\F^n_T),\\
&&	h(T) = \sup_{\F} h(T,\F)
\end{eqnarray*}
with the supremum ranging over all finite collections of measurable functions from $X$ to $[0,1]$ (existence of the limit was proved in \cite{DF}). For the conditional entropy of a collection $\F$ with respect to the collection $\G$ we put
\[
H(\F|\G)=H(\F\vee\G)-H(\G)=H_{\mu\times\lambda}(\A_\F|\A_\G).
\]
The following continuity assertion is true: for every $r\ge 1$ 
and $ \eps>0$ there is a $\delta>0$ such that if $\mathcal F$ and $\mathcal G$ have cardinalities at most $r$ and $\dist(\mathcal F, \mathcal G)<\delta$ 
then $|H(\F|\G)|< \eps$.

It is known that the above procedure leads to a common value of entropy $h(T)$ of a doubly stochastic operator for many reasonable choices of $H(\F)$ (see \cite{D} or \cite{DF} for details on axiomatic theory of metric entropy of doubly stochastic operators). We find the above version convenient for our purposes. Moreover, we see that since the definition depends only on bounded functions, in fact only on functions $X\to[0,1]$, the choice of the domain $L^p(\mu)$ of an operator does not affect the value of entropy.

The \emph{sequence entropy} of a measure preserving map $T$ with respect to a partition $\xi$ along a subsequence  $A=(i_n)_{n\in\N}$ of $\{0,1,2...\}$ is defined by
\[
h_A(T,\xi)=\limsup_{n\to\infty} \frac1n H(\bigvee_{k=1}^n T^{-i_k}\xi),
\]
and the \emph{sequence entropy} of $T$ is
\[
h_A(T)=\sup_{\xi} h_A(T,\xi),
\]
where the supremum ranges over all finite partitions of $X$.
Similarly, in operator case we formulate the following definition.
\begin{defn}
The \emph{sequence entropy} of a doubly stochastic operator~$T$ with respect to a collection of functions~$\F$ along a sequence  $A=(i_n)_{n\in\N}$ is defined as 
\[
h_A(T,\F)=\limsup_{n\to\infty} \frac1n H(\bigvee_{k=1}^n T^{i_k}\F),
\]
and the \emph{sequence entropy} of the operator~$T$ along a sequence  $A$ is given by
\[
h_A(T)=\sup_{\F} h_A(T,\F).
\]
\end{defn}
By a brief inspection of arguments in \cite{DF} we check that the value of $h_A(T)$ does not depend on the choice of the static entropy $H(\F)$ as long as the latter fulfils requirements of the axiomatic definition of entropy.

%====================== DENSITY OF ENTROPY ZERO ====================================

\section{Genericity of entropy zero}

\begin{lem}
Let $T$ be a doubly stochastic operator on $L^p(\mu)$. Let $Sf=\int fd\mu$. Then $(1-\frac1n)T+\frac1n S$ converges to $T$ in operator norm, hence also in strong operator topology.
\end{lem}
\begin{proof}
\[
\begin{split}
\pnorm{p}{\left((1-\frac1n)T+\frac1n S\right)f-Tf} = \frac1n\pnorm{p}{Tf-Sf} \\
\leq \frac1n \left(\pnorm{p}{Tf}+\pnorm{p}{Sf}\right) \leq \frac2n \pnorm{p}{f}
\end{split}
\]
This ends the proof for strong operator topology. For norm topology, it is enough to take supremum over $f\in L^p(\mu)$ with $\pnorm{p}{f}=1$.
\end{proof}

\begin{cor}\label{dense}
The set of doubly stochastic operators with zero entropy is dense in the strong operator topology and in the norm operator topology.
\end{cor}
\begin{proof}
For every $\alpha\in(0,1)$ the operator $(1-\alpha)T+\alpha S$ has zero entropy, because $\big((1-\alpha)T+\alpha S\big)^n f$ converges to a constant function for every~$f$.
\end{proof}

\begin{lem}\label{Gdelta}
The set of doubly stochastic operators with zero entropy is a $G_\delta$ set in the strong operator topology in $L^1(\mu)$.
\end{lem}
\begin{proof}
Fix $\eps>0$, a positive integer $n$ and a finite collection $\F$ of $r$ measurable functions $X\to[0,1]$.
We will show that the set 
\[
U(\eps,n,\F)=\left\{T \in \DS(\mu):\frac1nH(\F^n)<\eps\right\}
\]
is open.
Let $T\in U(\eps,n,\F)$ and let $\frac1nH(\F^n)=\gamma<\eps$. Choose $\delta>0$ so that $\dist(\G,\G')<n\delta$ guarantees that $H(\G|\G')+H(\G'|\G) < \eps-\gamma$ for all collections $\G,\G'$ containing $r$ functions. Consider an open set 
\[
C=\{P\in \DS(\mu): \forall k=0,...,n-1\ \forall f\in\F\ \pnorm{1}{PT^{k-1}f-T^kf}<\delta\}
\]
Note that if $P$ lies in $C$ we have 
\[
\begin{split}
\pnorm{1}{P^kf-T^kf} \leq \pnorm{1}{P^kf-P^{k-1}Tf} + \pnorm{1}{P^{k-1}Tf-P^{k-2}T^2f} + ...\\
... + \pnorm{1}{PT^{k-1}f-T^kf} < k\delta
\end{split}
\]
for every $f\in\F$ and $k=0,1,...,n-1$, because $P$ is a contraction. Then $H(T^kF|P^kF) <\eps$ implying that
\[
H(\bigvee_{k=0}^{n-1}P^kF|\bigvee_{k=0}^{n-1}T^kF) 
\leq \sum_{k=0}^{n-1} \big(H(P^kF|T^kF)\big) < n(\epsilon-\gamma)
\]
and
\[
\frac1n H(\bigvee_{k=0}^{n-1}P^kF) \leq \frac1n H(\bigvee_{k=0}^{n-1}T^kF) + \frac1n H(\bigvee_{k=0}^{n-1}P^kF|\bigvee_{k=0}^{n-1}T^kF) < \epsilon
\]
Consequently, $U(\eps,n,\F)$ is open. 

Since $L^1$ is separable we can choose a countable dense subset $\mathbb{F}$ of the set of all integrable functions with range in $[0,1]$. Now
\[
\mathcal{U}=\bigcap_m \bigcap_r \bigcap_{\F\in\mathbb{F}^r} \bigcap_N \bigcup_{n\geq N} U(\frac1m,n,\F)
\]
is a $G_\delta$ set. We will prove that this is exactly the set of operators having zero entropy.

If $T$ has entropy zero then it clearly belongs to this set.

On the other hand, let again $\eps>0$, $n\in\N$ and let $\F$ be a collection of~$r$ functions. For $T\in\mathcal{U}$ choose $\tilde\F\in\mathbb{F}^r$ so that $\dist(\F,\tilde\F)$ is small enough to ensure that $H(\F|\tilde\F)<\eps/2$. Then also $H(T^k\F|T^k\tilde\F)<\eps/2$ for $k\in\N$ and, consequently,
\[
H(\F^n_T|{\tilde\F}^n_T)\leq\sum_{k=0}^{n-1}H(T^k\F|T^k\tilde\F)<n\eps/2
\]
Therefore,
\[
H(\F^n_T)\leq H(\tilde\F^n_T)+H(\F^n_T|\tilde\F^n_T) <  H(\tilde\F^n_T) + n\eps/2
\]
For any $m$ we can find arbitrarily large $n$ for which $\frac1n H(\tilde\F^n_T)<\frac1m$, so taking $m>\eps/2$ we obtain the existence of a large $n$ such that $\frac1n H(\F^n_T)<\eps$. Finally, we obtain a subsequence of $\frac1nH(\F^n_T)$ tending to zero, but since the whole sequence is convergent, it converges to zero. The collection $\F$ was chosen freely, so the entropy of $T$ is zero.
\end{proof}

\begin{rem}
Note that it follows that the set is also a $G_\delta$ set in the norm topology, as it has fewer open sets than the strong operator topology.
\end{rem}

From lemma \ref{Gdelta} and corollary \ref{dense} it follows that:
\begin{thm}	\label{new_Rokhlin}
The set of doubly stochastic operators with zero entropy is residual in strong operator topology and in the norm topology of $L^1(\mu)$.
\end{thm}

%====================== DECOMPOSITION OF THE DOMAIN ================================

\section{Decompositions of the domain}
\newcommand{\uni}[1]{{{#1}_{\mathrm{uni}}}}
\newcommand{\cnu}[1]{{{#1}_{\mathrm{cnu}}}}
\newcommand{\rev}[1]{{{#1}_{\mathrm{rev}}}}
\newcommand{\aws}[1]{{{#1}_{\mathrm{aws}}}}

We recall basic facts on the Jacobs-de Leeuw-Glicksberg (JdLG-) decomposition and Nagy-Foia\c{s} (NF-) decomposition. Our sources of information are \cite{EFHN} in case of JdLG-decomposition and \cite{NF} for NF-decomposition. 

We say that a subspace $W$ of the space $L^p(\mu)$ is \emph{invariant} with respect to a doubly stochastic operator $T:L^p(\mu)\to L^p(\mu)$ if $TW\subset W$. 

\begin{defn}
A \emph{JdLG-decomposition of the space $L^p(\mu)$ associated with a doubly stochastic operator $T$} is a pair of $T$-invariant subspaces $\rev{E}$ and $\aws{E}$ such that $L^p(\mu) = \rev{E} \oplus \aws{E}$, where $\rev{E}$ is the range and $\aws{E}$ is the kernel of a unique minimal idempotent of the weak operator closure of the semigroup $\{I,T,T^2,...\}$.

The spaces $\rev{E}$ and $\aws{E}$ are called the \emph{reversible part} and the \emph{almost weakly stable part} of the space, respectively. 
\end{defn}
The existence of a unique minimal idempotent follows from the famous Ellis' theorem proved in \cite{E} and such decomposition can be obtained even in more general situations. It is clear from the definition that this decomposition is unique. For our purposes, the following characterizations will be of greatest importance.
\begin{thm}	\label{thm:criteria}
If $T$ is a doubly stochastic operator (or just a contraction) on $L^p(\mu)$ and $\rev{E} \oplus \aws{E}$ is its JdLG-decomposition then:
\begin{enumerate}
	\item $\rev{E}=\cllin\{f\in L^p(\mu): \exists\lambda\in\C, |\lambda|=1\ s.t.\ Tf=\lambda f\}$ (see \cite{EFHN}, thm.16.33),
	\item for every $f\in\aws{E}$ either $\{T^nf:n\in\N\}$ is not precompact in $L^p$ norm topology or $\inf_{n\in\N}||T^nx||=0$ (see \cite{EFHN}, thm.16.29).
\end{enumerate}
\end{thm}

Now assume that $\mathscr{H}$ is a Hilbert space and $T$ is a contraction of $\mathscr{H}$. We say that $W$ \emph{reduces} $\mathscr{H}$ if both $W$ and $W^\perp$ are invariant (or, equivalently, $W$ is invariant with respect both to $T$ and $T^*$). An operator $T$ on a Hilbert space $\mathscr{H}$ is \emph{completely non-unitary (c.n.u.)} if there is no reducing 
subspace on which it is unitary.

\begin{defn}
Let $T$ be a contraction of a Hilbert space $\mathscr{H}$. A \emph{NF-decomposition} is a decomposition $\mathscr{H}= \uni{\mathscr{H}}\oplus \cnu{\mathscr{H}}$ into an orthogonal sum of two subspaces reducing $T$, such that $T|{}_\uni{\mathscr{H}}$ is unitary and $T|{}_\cnu{\mathscr{H}}$ is c.n.u. The spaces $\uni{\mathscr{H}}$ and $\cnu{\mathscr{H}}$ are the \emph{unitary part} and the \emph{completely non-unitary part} of $\mathscr{H}$, respectively, and the restrictions of $T$ to these subspaces are the \emph{unitary part} and the \emph{completely non-unitary part} of $T$ (compare \cite{NF}, thm3.2).
\end{defn}

Just as in the previous case,  the decomposition is unique. Note that for a unitary operator the c.n.u. part is trivial, while for a c.n.u. operator the unitary part is trivial.
A doubly stochastic operator on $L^2(\mu)$ is never c.n.u., because the subspace consisting of constants is a reducing subspace on which $T$ is unitary.
Roughly saying, the unitary part captures all the invertibility of a doubly stochastic operator in the sense that if there is an invertible pointwise factor then there is a closed invariant subspace of $L^2(\mu)$. Moreover, if $T$ is a Koopman operator of an invertible map then $\cnu{L^2(\mu)}=\{0\}$. In other cases the c.n.u. part is non-trivial and it can be the whole $L_0^2(\mu)$---the orthogonal complement of constants---even for pointwise operators, e.g., for a one-sided shift. It is easy to see that if $\rev{\mathscr{H}} \oplus \aws{\mathscr{H}}$ is a JdLG-decomposition and $\uni{\mathscr{H}}\oplus \cnu{\mathscr{H}}$ is NF-decomposition of a Hilbert space $\mathscr{H}$ then 
\[
\rev{\mathscr{H}} \subset \uni{\mathscr{H}}\qquad \mathrm{and} \qquad \aws{\mathscr{H}} \supset \cnu{\mathscr{H}}.
\]

%====================== DISCRETE SPECTRUM VERSUS NULLITY ================================

\section{Discrete spectrum versus nullity}

One of the basic definitions in ergodic theory is that of a discrete spectrum of a measure preserving transformation. In more general setup it sounds as follows:
\begin{defn}
 A power bounded operator $T$ on a Banach space $V$ has \emph{discrete spectrum} if 
\[
V=\cllin\{v\in V: \exists \lambda\in\C,\ |\lambda|=1, \ s.t.\  Tv=\lambda v\}
\]
\end{defn}
Clearly, it is applicable to the case of doubly stochastic operators, because $\pnorm{p}{T^n}=1$ for all $n\in\N$. Similarly to the case of pointwise maps, we will say that:
\begin{defn}
A doubly stochastic operator $T$ is \emph{null} if $h_A(T)=0$ for every sequence $A=\{t_1<t_2<t_3<...\}$ of positive integers.
\end{defn}
We now state the theorem which we aim to generalize.

\begin{thm}[Kushnirenko \cite{K}]	\label{Kushnirenko}
A transformation $T$ has discrete spectrum if and only if $T$ is null.
\end{thm}

This equivalence fails for doubly stochastic operators. For example $Tf=\int f\,d\mu$ does not have discrete spectrum, but $h_A(T)=0$ for every sequence $A$, which shows that nullity does not imply discrete spectrum. 
On the other hand, a contraction may have a linearly dense set of eigenvectors, though its unitary part is reduced to constants. Indeed, let $R$ be an irrational rotation on the circle~$S^1$, $Rz=\alpha z$ and let $Tf=\frac12 f\circ R + \frac12\int f\,d\lambda$, where $\lambda$ is the Haar probability measure. Clearly, each $e_n(z)=z^n$ is a ``c.n.u eigenfunction'' for eigenvalue $\frac12 \alpha^n$.
These examples show that for a general doubly stochastic operators null 
sequential entropy does not imply anything about existence or density of 
eigenvectors.

\begin{thm}	\label{thm:main}
The following conditions are equivalent for a doubly stochastic operator $T:L^p(\mu)\to L^p(\mu)$:
\begin{enumerate}
	\item $T$ is null,
	\item $L^p(\mu)=V \oplus W$, where $T$ has discrete spectrum on $V$ and $\lim_{n\to\infty} \pnorm{p}{T^nf} = 0$ for every $f\in W$,
	\item if $\rev{E} \oplus \aws{E}$ is a JdLG-decomposition of $L^p(\mu)$ associated with $T$, then $\lim_{n\to\infty} \pnorm{p}{T^nf} = 0$ for every $f\in \aws{E}$.
\end{enumerate}
If $p=2$ then each of the above conditions is equivalent to:
\begin{enumerate}
\setcounter{enumi}{3}
	\item if $\uni{\mathscr{H}} \oplus \cnu{\mathscr{H}}$ is the NF-decomposition of $L^2(\mu)$ with respect to~$T$, then $T$ has discrete spectrum on $\uni{\mathscr{H}}$ and $\pnorm{2}{T^nf}$ converges to~0 for every $f\in\cnu{\mathscr{H}}$.
\end{enumerate}
\end{thm}

Before the proof of this theorem let us state several lemmas. Walking in the footsteps of Kushnirenko we translate a part of his proof to the operator case.
\begin{lem}
Let $A=\{t_1<t_2<t_3<...\}$ be a sequence of positive integers.
Then $h_A(T)=0$ if and only if $h_A(T,\{f\})=0$ for every $f$.
\end{lem}
\begin{proof}
By the definition of entropy, $h_A(T)\geq \sup_{f} h_A(T,\{f\})$, where the supremum ranges over the set of all measurable functions $f:X\to[0,1]$. On the other hand, it follows easily from subadditivity of entropy that if $h_A(T,\{f\})=0$ for every $f$ then $h_A(T,\F)=0$ for every collection $\F$ of measurable functions.
\end{proof}

\begin{lem}
Let $f:X\to[0,1]$ be a measurable function.
Then the closed orbit $\bar{O}_T(f)=\overline{\{T^nf:n=0,1,2,...\}}$ is compact if and only if $h_A(T,\{f\})=0$ for all sequences $A=\{t_1<t_2<t_3<...\}$.
\end{lem}
\begin{proof}
Fix $\eps >0$ and $A=\{t_n:n\in\N\}$. Let us abbreviate $\A_{T^n f}$ by $\xi_n$.
Let $\mathcal{Z}$ be the set of all countable measurable partitions of $X$ with $H_\mu(\xi)<\infty$. 
We remind that in the context of entropy the natural distance between two partitions $\xi$ and $\zeta$ belonging to $\mathcal{Z}$ is given by the Rokhlin metric:
\[
\rho(\xi,\zeta) = H_\mu(\xi|\zeta) + H_\mu(\zeta|\xi).
\]

If $\bar{O}_T(f)$ is compact in $L^p(\mu)$ then $\{\A_g:g\in \bar{O}_T(f)\}$ is compact in~$\mathcal{Z}$. Indeed, $\lim_{n\to\infty} \pnorm{p}{g_n-g}=0$ implies that $\lim_{n\to\infty}\mu(A_{g_n}\symdiff A_g)=0$, so a convergent subsequence of $(g_n)_{n\in\N}$ induces a convergent subsequence of $(\A_{g_n})_{n\in\N}$. The closure of $\{\xi_{t_n}:n\in\N\}$ is also compact, hence it contains an $\eps$-dense set $\{\xi_{t_n}:1\leq n\leq N\}$.
Therefore, 
\begin{eqnarray*}
h_A(T,\{f\}) & = & \limsup_{n\to\infty} \frac1n H(\bigvee_{i=1}^n \xi_{t_i}) \\
& \leq & \limsup_{n\to\infty} \frac1n \Big(H(\bigvee_{i=1}^N \xi_{t_i}) + \sum_{j=N+1}^n H(\xi_{t_j}| \bigvee_{i=1}^{N} \xi_{t_i}\Big) \leq \eps
\end{eqnarray*}

Conversely, let $\bar{O}_T(f)$ be non-compact, i.e., for some $\gamma>0$ there is a sequence $g_n\in\bar{O}_T(f)$ satisfying $\pnorm{p}{g_m-g_n}\geq\gamma$ for any $m\not=n$. Then $\pnorm{1}{g_m-g_n}\geq \gamma^p$ implying that $\mu(A_{g_m}\symdiff A_{g_n})$ are separated from 0 for all pairs $m\not=n$. Therefore, $\rho(\A_{g_m},\A_{g_n}) >\eps$ for some $\eps>0$ and all $m\not=n$. We now inductively choose a sequence $\xi_n$ satisfying $\limsup_{n\to\infty} \frac1n H(\bigvee_{i=1}^n \xi_i)\geq \delta$ for some $\delta>0$. It suffices to ensure that $H(\xi_n|\bigvee_{i=1}^{n-1}\xi_i)\geq\delta$ for all $n>1$. The appropriate number $\delta>0$ is chosen with use of the following lemma. This is a slight modification of the one proved in \cite{K}, so we leave the proof as an exercise. We only hint that if $g\in \bar{O}_T(f)$ then $\int g\,d\mu = \int f\,d\mu $.
\begin{lem}	\label{Kush}
For every $\eps>0$ there exists a $\delta>0$ such that for any finite partition $\xi$ of $X\times [0,1]$ and any partition $\zeta=\A_g$ with $g\in \bar{O}_T(f)$ the inequality $H(\zeta|\xi)<\delta$ implies the existence of $\xi'\prec\xi$ such that $\rho(\zeta,\xi')<\frac\eps2$.
\end{lem}
Let $\xi_1=\A_{g_1}$. Note that there are only two partitions coarser than~$\xi_1$, namely, the trivial partition and $\xi_1$ itself. By the above lemma, if $H(\A_{g_m}|\xi_1)<\delta$ for all $m$, then each $\A_{g_m}$ is $\eps/2$-close (in Rokhlin metric) to one of these two partitions, which contradicts the fact that all partitions $\A_{g_m}$ are $\eps$-separated. Hence, there is a partition $\xi_2$ of the form $\A_{g_m}$, such that $H(\xi_2|\xi_1)\geq\delta$.
For the induction step, suppose that we have already found partitions $\xi_1,\xi_2,...,\xi_{n-1}$, chosen from $(\A_{g_m})_m$.
Let $\xi=\bigvee_{i=1}^{n-1}\xi_i$. Similarly as before, there are only finitely many partitions $\xi'$ coarser than $\xi$, and each $\A_{g_m}$satisfying $H(\A_{g_m}|\xi)<\delta$ is $\eps/2$-close to some $\xi'$ (by lemma \ref{Kush}). Since all partitions $\A_{g_m}$ are $\eps$-separated, the tail of the sequence $(\A_{g_m})_m$ contains some element $\xi_n$ satisfying $H(\xi_n|\xi)\geq\delta$.
(Note that it is irrelevant that $\xi'$ need not be of the form $\A_g$ for any $g$.)
\end{proof}
Putting these pieces together we obtain:
\begin{cor}	\label{cor:compact_orbits}
A doubly stochastic operator $T$ on $L^p(\mu)$ is null if and only if the closed orbit $\bar{O}_T(f)$ is compact for every $f\in L^p(\mu)$.
\end{cor}
\begin{lem}	\label{lem:compact_orbits_on_rev}
If $T$ is a doubly stochastic operator then the closed orbit $\bar O_T(f)$ is compact for any $f\in \rev{E}$.
\end{lem}
\begin{proof}
If $\{e_n:n\in\N\}$ is a linearly dense set  of eigenvectors, then the orbit of $f=\sum_{n=1}^\infty a_n e_n$ is a subset of the compact set $\{\sum_{n=1}^\infty b_n e_n: |b_n|\leq |a_n|\}$. 
\end{proof}
\begin{cor}
If $T$ is a doubly stochastic operator with discrete spectrum then $T$ is null.
\end{cor}

\begin{proof}[Proof of thm.~\ref{thm:main}]
The third condition is just a reformulation of the second one, as $\rev{E}$ is exactly the set spanned by the eigenvectors of $T$. We concentrate on proving that the first two statements are equivalent. We remark that for a doubly stochastic operator we have $\inf_{n\in\N} ||T^nf||=\lim_{n\to\infty} ||T^nf||$.

Assume that $T$ is null. Let $V=\rev{E}$, $W=\aws{E}$, where $\rev{E} \oplus \aws{E}$ is a JdLG-decomposition of $L^p(\mu)$ associated with $T$. By definition of~$\rev{E}$, operator $T$ has discrete spectrum on $V$. By corollary~\ref{cor:compact_orbits} nullity is equivalent to all orbits being precompact, which in view of theorem~\ref{thm:criteria} forces all orbits in $\aws{E}$ to approach zero. 

On the other hand, assuming the second condition we obtain from lemma~\ref{lem:compact_orbits_on_rev} that all orbits of elements of $V$ are precompact. Clearly, also orbits of elements of $W$ are precompact (simply because they are convergent sequences), so all functions from $V \oplus W$ have precompact orbits and it follows from corollary~\ref{cor:compact_orbits} that $T$ is null.

To prove the last statement note that it is always true that $\rev{E} \subset \uni{\mathscr{H}}$ and  $\cnu{\mathscr{H}}\subset \aws{E}$. If $T^nf$ converges to 0 then $f\in\cnu{\mathscr{H}}$, hence these two decompositions coincide. On the other hand, if $T$ has has discrete spectrum on $\uni{\mathscr{H}}$ then $\uni{\mathscr{H}}=\rev{E}$, so again the decomopositions coincide.
\end{proof}

An entire class of well-studied operators is 
guaranteed to be null:
\begin{defn}
	A doubly stochastic operator $T$ on $L^2(\mu)$ is \emph{quasi-compact}, if 
	there exists a 
	direct sum decomposition $L^2(\mu)=F\oplus H$ and $r<1$, such that:
	\begin{enumerate}
		\item $F$ and $H$ are closed invariant subspaces of $L^2(\mu)$,
		\item $\dim(F)<\infty$ and all eigenvalues of $T|_F$ have modulus 
		larger than $r$,
		\item the spectral radius of $T|_H$ is smaller than $r$. 
	\end{enumerate}
\end{defn}
\begin{thm}
	If $T$ is a quasi-compact doubly stochastic operator, then $T$ is null.
\end{thm}
\begin{proof}
	By corollary \ref{cor:compact_orbits}, it suffices to show that the 
	closed orbit of $f$ is compact for every $f\in L^2(\mu)$. Let 
	$L^2(\mu)=F\oplus H$ be the decomposition from the definition of 
	quasi-compactness, and let $f=g+h$, where $g\in F$, $h\in H$. Since $F$ is 
	a finite-dimensional invariant subspace, the closure of the set 
	$\set{T^ng:n\in\N}$ is compact. On the other hand, by the spectral radius 
	formula, we know that for every $\eps>0$ we have $\pnorm{2}{T^n h}/\pnorm{2}{h} = 
	O(e^{n\eps} r^n )$. If we take $\eps<\log\frac{1}{r}$, this implies that 
	$\pnorm{2}{T^nh}$ converges to $0$, and thus the orbit $\bar{O}_T(f)$ is compact.
\end{proof}
\begin{rem} \label{rem:radius}
The above argument can be repeated to prove the following fact: if $T$ is a doubly stochastic operator and $V \oplus W$ is the decomposition into unitary and c.n.u. parts, such that $T|{}_V$ has discrete spectrum and $T|{}_W$ has spectral radius strictly smaller than 1, then $T$ is null. 
\end{rem}

To illustrate that having the spectral radius  equal to 1 does not guarantee non-nullity we will show an example of a null c.n.u. operator which has spectral radius equal to $1$ when considered on the orthogonal complement of constants. In addition, the operator has no non-trivial eigenfunctions.
\begin{example}
    For every $n$, let $X_n=[0,1]$, let $\mu_n$ be the Lebesgue measure and let $R_n$ be an operator defined on $L^2(X_n)$ by the formula $(R_n f)(x)=\frac{n-1}{n}f(x)+\frac{1}{n}\int fd\mu_n$. Let $X=X_1\times X_2\times\ldots$ and let $\mu$ be the product measure. Define the operator $R$ first on functions of the form $f(x)=f_1(x_1)f_2(x_2)\cdots f_n(x_n)$, $x=(x_k)_{k\in\N}$, $f_k\in L^2(\mu_k)$, as $Rf=R_{1}f_1R_{2}f_2\cdots R_{n}f_n$. Since the functions of this form are linearly dense in $L^2(\mu)$, $R$ extends to a doubly stochastic operator on all of $L^2(\mu)$. Let $S$ be a doubly stochastic operator on $L^2(\mu)$ induced by the left shift, i.e., $(Sf)(x_1,x_2,x_3,\ldots)=f(x_2,x_3,\ldots)$. Finally, let $T=SR$. For any $k$, if we take any $g\in L^2(\mu_k)$ and define $f\in L_0^2(\mu)$ as $f(x)=g(x_k)$, then 
				\begin{eqnarray}
				T^mf(x)&=&(R_{m+k-1}R_{m+k-2}\cdots R_kg)(x_{m+k-1})	\nonumber\\
				&=&\frac{k-1}{m+k-1}g(x_{m+k-1}) + \frac{m}{m+k-1} \int g d\mu_k.	\label{m_iterate}
				\end{eqnarray}
		Moreover,
		\[
		T^mf(x)= \prod_{k=1}^n T^mf_k (x_{m+k-1})
		\]
		for $f(x)=f_1(x_1)f_2(x_2)\cdots f_n(x_n)$, $x=(x_k)_{k\in\N}$, $f_k\in L^2(\mu_k)$.
		Below we verify the afforementioned properties of $T$.
    \begin{itemize}
        \item $T$ is null, because the orbit of every function is precompact. To see this, first note that if  $f(x)=f_1(x_1)f_2(x_2)\cdots f_n(x_n)$, $x=(x_k)_{k\in\N}$, then $T^mf$ converge to the constant function equal to $\int f_1\mu_1\cdots\int f_nd\mu_n=\int fd\mu$. Since such functions are linearly dense in $L^2(\mu)$, the same convergence to a constant holds for every function in $L^2(\mu)$.
        \item If $f$ is any function orthogonal to the constants, the images of $f$ under iterations of $T$ converge to $0$, which means that $T$ is c.n.u. on the space of such functions.
        \item If $f\in L_0^2(\mu)$ has the form $f(x)=g(x_k)$ for $g\in L^2(\mu_k)$, $x=(x_k)_{k\in\N}$, then by (\ref{m_iterate}) we have $\pnorm{2}{T^mf} =\frac{k-1}{m+k-1}\pnorm{2}{f}$. Since $k$ was arbitrary, the norm of $T^m$ is equal to $1$. By the spectral radius formula, we conclude that $P$ has spectral radius equal to one.
        \item     Finally, suppose that for some $\lambda\neq 0$ there is a function $f\in L_0^2(\mu)$ such that $Tf=\lambda f$. If $g$ is any function dependent only on the first $m$ coordinates, then since $T^m(f)$ only depends on coordinates from $m+1$ onwards, $T^mf$ and $g$ are orthogonal. As $T^mf=\lambda^mf$, this means that $f$ is orthogonal to every function dependent only on the first $m$ coordinates. As $m$ was arbitrary, it follows that $f$ is orthogonal to every function dependent only on finitely many coordinates, but since such functions are dense in $L_0^2(\mu)$, we conclude that $f=0$.
    \end{itemize}
\end{example}

%====================== REPRESENTATION THEOREM ================================

\section{Representation theorem}
The celebrated theorem of Halmos and von Neumann states that every ergodic measure preserving dynamical system with discrete spectrum is isomorphic to a Kronecker system, i.e., a rotation of a compact abelian group. Since by Kushnirenko's theorem discrete spectrum is equivalent to being null, it is natural to search for an analogous representation theorem for null systems in case of doubly stochastic operators.
Consider the following example.

\begin{example}
Let $X$ be the annulus $S^1 \times I$, $I=[0,1]$, with the product of the Haar and Lebesgue measures. For $\alpha\in[0,1)$ define an operator by
\[
Tf(z,x) = \int f(ze^{2\pi i\alpha},x)\,d\lambda(x),
\]
namely, a product of a rotation by $\alpha$ and a trivial integral operator. Clearly, it is null with the reversible part $E_{rev}$ of the JdLG-decomposition consisting of functions constant on fibers $\{z\}\times I$. It is essentially a rotation of a circle with each point of a circle pumped to the unit interval and the transition probability between these intervals is given by the uniform distribution.

It is easy to modify the above example to see that we cannot expect a representation in a form of a product action. For instance, on the same space $X=S^1 \times I$ let us define  the measure:
\[
\nu(A)=\lambda\times \lambda \Big((A\cap \big(N\times I\big)\Big) + \lambda\times \delta_0 \Big(A\cap \big(N^c\times I\big)\Big),
\]
where $N=\{z:\arg(z)\in[0,\pi)\}$.
An appropriate modification of $T$, namely
\[
Tf(z,x) = \int f(ze^{2\pi i\alpha},x)\,d\lambda(x)\cdot\ind_{R_\alpha^{-1}N}(z) + f(ze^{2\pi i\alpha},0)\cdot\ind_{R_\alpha^{-1}N^c}(z),
\]
again gives a null operator with the same reversible part~$E_{rev}$.
\end{example}

Let $\rev{\Sigma}=\{A\in\Sigma:\ind_A\in \rev{E}\}$. Clearly, $\rev{\Sigma}$ is a sub-$\sigma$-algebra of $\Sigma$ and it is known that $\rev{E}=L^1(X,\rev{\Sigma},\mu)$. 
\begin{defn}Let $T$ be a Markov operator on $L^1(\mu)$.
\begin{itemize}
	\item $T$ is a \emph{Markov embedding} if it is a lattice homomorphism (i.e., $|Tf|=T|f|$ for every $f\in L^1(\mu)$) or, equvalently, there is a Markov operator $S$ such that $ST$ is an identity.
	\item $T$ is a \emph{Markov isomorphism} if it is a surjective Markov embedding.
\end{itemize}
\end{defn}

Markov embeddings preserve lattice operations of taking maximum or minimum and transform characteristic functions into characteristic functions. Therefore, they are in a natural one-to-one correspondence with homomorphisms of measure algebras, and in case of Lebesgue spaces, they are just Koopman operators of measure preserving transformations. Hence, for a Lebesgue space a Markov isomorphism can alwasy be replaced by a point isomorphism of dynamical systems.

\begin{thm}	\label{new_HvN}
An ergodic doubly stochastic operator $T$ on $(X,\Sigma,\mu)$ is null if and only the following two conditions are satisfied:
\begin{enumerate}
	\item the action of $T$ on $\rev{E}=L^1(X,\rev{\Sigma},\mu)$ is Markov isomorphic to a rotation $R$ of a compact abelian group $G$ with Haar measure~$\lambda$,
	\item 
	\[
	\lim_{n\to\infty} \pnorm{2}{T^nf - T^nE(f|\rev{\Sigma})}=0
	\]
	for every $f$.
\end{enumerate}

Furthermore, if $X$ is a Lebesgue space and $P_T$ is a transition probability inducing $T$, then the Markov isomorphism becomes a point isomorphism of dynamical systems, there is a measure-preserving map $\pi:X\to G$ satisfying $T(g\circ\pi)=g\circ R\circ \pi$ for every $g\in L^1(G,\lambda)$ and $P_T(x,\cdot)$ is supported on $\pi^{-1}R\pi(x)$.
\end{thm}

\begin{proof}
The fact that the action of $T$ on $\rev{E}$ is Markov isomorphic to a rotation of a compact abelian group is an extension of the Halmos-von Neumann theorem formulated as Theorem 17.6 in \cite{EFHN}.

Since $E(f|\rev{\Sigma})\in\rev{E}$, the function $f-E(f|\rev{\Sigma})$ belongs to~$\aws{E}$. If $h_A(T)=0$ for every sequence $A$, then \[
\lim_{n\to\infty}\pnorm{2}{T^n(f - E(f|\rev{\Sigma}))}= 0,
\]
by thm.~\ref{thm:main}. Conversely, this condition guarantees that $h_A(T)=0$ for every $A$, because $E\big(f|\rev{\Sigma}\big)=0$ for $f\in\aws{E}$.

The existence of a factor map $\pi$ in case of a Lebesgue space follows from the fact that a natural injection $\rev{\Sigma}\to\Sigma$ is a measure algebra homomorphism, hence it comes from a measure-preserving map.

Finally, if $X$ is a Lebesgue space then for any~$g\in L^1(\G,\lambda)$, 
\[
\int g\circ\pi(y) P_T(x,dy)=T(g\circ\pi)(x)=g(R\pi(x)).
\]
Taking $g=\ind_{\{R\pi(x)\}}$ we obtain that $P_T(x,\cdot)$ is supported on $\pi^{-1}R\pi(x)$.
\end{proof}

%==================================================================================

\end{document}